\documentclass[reqno,b5paper]{amsart}

\usepackage[T2A]{fontenc}
\usepackage[english]{babel}
\usepackage{graphicx}
\usepackage{amsmath}
\usepackage{amsfonts}
\usepackage{amssymb}
\usepackage{amsthm}

\usepackage{enumerate}
\usepackage[mathscr]{eucal}

\theoremstyle{plain}
\newtheorem{theorem}{Theorem}

\newtheorem{prop}[theorem]{Proposition}

\theoremstyle{remark}

\newtheorem*{remark}{Remark}

\theoremstyle{definition}

\begin{document}

\title[On characterization of integrable
sesquilinear forms]{On characterization of integrable sesquilinear
forms}

\author[A.N. Sherstnev --- O.E. Tikhonov]{A.N. Sherstnev* --- O.E. Tikhonov**}

\newcommand{\acr}{\newline\indent}

\address{\llap{*} Department of Mathematical Analysis  \acr
Faculty of Mechanics and Mathematics \acr Kazan Federal University
\acr Kremlyovskaya St., 18 \acr Kazan 420008 \acr RUSSIAN
FEDERATION}

\email{anatolij.sherstnev@gmail.com}

\address{\llap{**}
Department of Mathematical Statistics  \acr Faculty of Computer
Science and Cybernetics \acr Kazan Federal University \acr
Kremlyovskaya St., 18 \acr Kazan 420008 \acr RUSSIAN FEDERATION}

\email{Oleg.Tikhonov@ksu.ru}

\thanks{The work
of the second author was supported by the Ministry of Education and
Science of the Russian Federation (government contract No.
02.740.11.0193)}



\begin{abstract}
We give necessary and sufficient condition for a sesquilinear form
to be integrable with respect to a faithful normal state on a von
Neumann algebra.
\end{abstract}

\maketitle

The fundamental solution to the problem of constructing a
noncommutative $L_1 (\varphi )$-space associated with a faithful
normal semifinite weight $\varphi$ on a von Neumann algebra
$\mathcal{M}$ was obtained in 1972--78. This space was realized as
a space of ``integrable'' sesquilinear forms defined on a ``lineal
of weight'' and ``affiliated'' with $\mathcal{M}$. In the next years
this approach was thoroughly developed (see the survey \cite{TruShe}
and the monograph \cite{She1}). For the other approaches to the
integration with respect to weights and states we refer the reader
to the surveys \cite{TruShe}, \cite{PX} and the recent paper
\cite{HJX}.

It is well known that a bounded linear operator on a Hilbert space
is nuclear if and only if it has finite matrix trace (see for
instance \cite[Theorem III.8.1]{GK}). In the present paper we
examine a problem whether certain analogue of that assertion holds
for integrable sesquilinear forms.

In what follows, $H$ is a Hilbert space with the scalar product
denoted by $\langle \cdot , \cdot \rangle$.

Let $\varphi$ be a faithful normal semifinite weight on a von
Neumann algebra $\mathcal{M}$ of operators on $H$ (see, e.\,g.,
\cite{Ta2}),
$
\mathfrak{m}^+_{\varphi} =
\{ x \in \mathcal{M}^+ \colon \varphi (x) < +\infty \}$,
\,
$\mathfrak{m}_{\varphi}^{\text{sa}} = \mathfrak{m}^+_{\varphi} -
\mathfrak{m}^+_{\varphi} .
$
It is well known that the formula
$$
\|x\|_{\varphi} \equiv \inf \{\varphi (x_1+x_2) \colon x=x_1-x_2; \
x_1 , x_2 \in \mathfrak{m}^+_{\varphi} \}
$$
determines a norm $\| \cdot \|_{\varphi}$ on
$\mathfrak{m}_{\varphi}^{\text{sa}}$. By
$L_1(\varphi)^{\text{sa}}$ we will denote the corresponding completion of
$\mathfrak{m}_{\varphi}^{\text{sa}}$.

The linear subspace of $H$
$$
D_{\varphi} \equiv \{ f\in H \colon \exists \lambda >0 \ \forall x
\in \mathcal{M}^+ \ ( \langle xf ,f \rangle \leqslant \varphi(x))\}
$$
was introduced and called \emph{the lineal of weight} in
\cite{She2}. Clearly, if $\varphi$ is represented in the form
$$
\varphi = \sum_{i\in I} \langle \cdot f_i , f_i \rangle, \quad f_i
\in H,  \eqno(1)
$$
then $f_i \in D_{\varphi}$ ($i\in I$).

The real Banach space $L_1(\varphi)^{\text{sa}}$ can be realized by
hermitian sesquilinear forms defined on $D_{\varphi}$. Namely, if
$\widetilde{x} \in L_1(\varphi)^{\text{sa}}$ and $(x_n)$ is a Cauchy
sequence in the normed space
$(\mathfrak{m}_{\varphi}^{\text{sa}} , \| \cdot \|_{\varphi})$, which
determines the element $\widetilde{x}$ of the completion, then the
formula
$$
a_{\widetilde{x}}(f,g)=\lim_n \langle x_n f , g \rangle,\qquad
f,g\in D_{\varphi} ,
$$
correctly defines a hermitian sesquilinear form $a_{\widetilde{x}}$.
The sequence $(x_n)$ is called \emph{defining for}
$a_{\widetilde{x}}$.
Also, since
$|\varphi (x)| \le \|x\|_{\varphi}$ for any
$x \in \mathfrak{m}_{\varphi}^{\text{sa}}$,
the formula
$$
\varphi (a_{\widetilde{x}} ) = \lim_n \varphi (x_n )
$$
correctly defines the value $\varphi (a_{\widetilde{x}} )$ which is
called \emph{the integral} (or \emph{the expectation}) of the
sesquilinear form $a_{\widetilde{x}}$ with respect to $\varphi$.
Accordingly, such sesquilinear forms are called \emph{integrable}.
Moreover, the main result of \cite{She2} (Theorem 2) says that the
map $\widetilde{x} \mapsto a_{\widetilde{x}}$
($\widetilde{x} \in L_1(\varphi)^{\text{sa}}$) is injective (see
also \cite[Theorem 16.7]{She1}, \cite[Theorem 1]{TruShe}). Thus,
$L_1(\varphi)^{\text{sa}}$ is meaningfully described as a real
Banach space of integrable sesquilinear forms. The cone
$L_1(\varphi)^+$ of integrable positive sesquilinear forms induces
a natural order structure in $L_1(\varphi)^{\text{sa}}$. The space
$L_1(\varphi)$ is defined as a certain complexification of
$L_1(\varphi)^{\text{sa}}$ \cite[16.11]{She1}, \cite[1.5]{TruShe},
and the notion of the integral is extended to sesquilinear forms in
$L_1(\varphi)$. The following proposition gives an ``explicit'' form
of such integral.

\begin{prop}[{\cite[Proposition 17.11]{She1}}]
Let
$$
\varphi = \sum_{i\in I} \langle \cdot f_i , f_i \rangle , \quad f_i
\in H,  \eqno(1)
$$
be a faithful normal semifinite weight on a von Neumann algebra
$\mathcal{M}$ and $a\in L_1(\varphi)$. Then
$$
\varphi(a)=\sum_{i \in I} a(f_i,f_i),  \eqno(2)
$$
where the series in $(2)$ converges absolutely and its sum does not
depend on the choice of representation of $\varphi$ in the form
$(1)$.
\end{prop}

In \cite[page 166]{She1}, the following problem was posed: does the
converse to Proposition 1 hold? The theorem below gives an
affirmative answer to the question in the special case of normal
states.

\begin{theorem}
Let $\varphi$ be a faithful normal state on a von Neumann algebra
$\mathcal{M}$. For a sesquilinear form \, $a$ \, defined on $D_{\varphi}$,
the following conditions are equivalent:

\smallskip
$\mathrm{(i)}$
$a \in L_1(\varphi)$,

\smallskip
$\mathrm{(ii)}$
for any representation
$\varphi = \sum\limits_{i\in I} \langle \cdot f_i , f_i \rangle$,
the series $\sum\limits_{i \in I} a(f_i,f_i)$ converges absolutely
and the sum does not depend on the representation of $\varphi$.
\end{theorem}

\begin{proof}
By virtue of Proposition 1, it suffices to prove
$\mathrm{(ii)} \implies \mathrm{(i)}$. Moreover, it is clear that we
can restrict ourselves to the case when $a$ is hermitian.

So, let $\varphi$ be a faithful normal state on $\mathcal{M}$ and a
hermitian sesquilinear form $a$ on $D_{\varphi}$ satisfy
$\mathrm{(ii)}$.

Denote by $Y$ the Banach space of hermitian $\sigma$-weakly
continuous functionals $\psi$ on $\mathcal{M}$ such that
$-\lambda \varphi \le \psi \le \lambda \varphi$ for some
$\lambda \ge 0$, supported with the norm
$$
\| \psi \|^\varphi = \inf \{ \lambda \ge 0 \colon -\lambda \varphi
\le \psi \le \lambda \varphi \}.
$$
Observe that if $-\lambda \varphi \le \psi \le \lambda \varphi$ then
$0 \le \frac12 (\lambda \varphi - \psi ) \le \lambda \varphi$,
$0 \le \frac12 (\lambda \varphi +\psi ) \le \lambda \varphi$ and
$\psi =\frac12(\lambda\varphi +\psi )-\frac12(\lambda\varphi -\psi)$.
Therefore the space $Y$ is generated by its positive part $Y^+$. One
can verify in a standard way that the restriction operation
$\Psi \mapsto \Psi|_{\mathcal{M}^{\text{sa}}}$ determines an
isometric and order isomorphism between the Banach conjugate space
$(L_1(\varphi)^\text{sa})^*$ and $Y$; and we will identify these spaces.

Associate with the form $a$ the linear functional $F_a$ on $Y$ in
the following way.

a) If $0 \le \psi\le \lambda \varphi$ and
$\psi=\sum\limits_{i \in I} \langle \cdot g_i , g_i \rangle$ then
$g_i\in D_{\varphi}$, and we set
$$
F_a (\psi) \equiv \sum_{i \in I} a(g_i, g_i).
$$
The value $F_a (\psi)$ is defined correctly. Indeed, let
$\psi=\sum\limits_{j \in J} \langle \cdot h_j , h_j \rangle$ be
another representation of $\psi$. Then, assuming that $\lambda =1$
for laying out simplification, we have
$$
\varphi= \sum_{i \in I} \langle \cdot g_i , g_i \rangle
+ \sum_{k \in K} \langle \cdot l_k , l_k \rangle
 = \sum_{j \in J} \langle \cdot h_j , h_j \rangle
+ \sum_{k \in K} \langle \cdot l_k , l_k \rangle
$$
for some $l_k \in H$. Consequently,
$$
\sum_{i \in I} a( g_i , g_i ) + \sum_{k \in K} a(l_k , l_k )
 = \sum_{j \in J} a( h_j , h_j ) + \sum_{k \in K} a( l_k , l_k ),
$$
hence, $\sum\limits_{i \in I} a( g_i , g_i )
= \sum\limits_{j \in J} a( h_j , h_j )$.

b) The functional $F_a$ defined above on $Y^+$ is additive and
positively homogeneous, therefore it can be uniquely extended to the
linear functional on $Y$.

It is easily seen that $F_a$ has the property:
$$
\text{\emph{if\,\,
$\psi, \psi_n \in Y^+$\, and\, $\psi = \sum\limits_{n=1}^{\infty} \psi_n$\,
then\, $F_a(\psi) = \sum\limits_{n=1}^{\infty} F_a (\psi_n)$.}} \eqno(3)
$$
 It follows,
in  particular, that $F_a$ is bounded. Indeed, it suffices to prove
that
$$
\sup \{ |F_a(\psi)| \colon 0 \le \psi \le \varphi \} < \infty .
$$
If the latter were false, there would exist a sequence $(\psi_n)$
such that $0 \le \psi_n \le \varphi$ and $|F_a (\psi_n)| \ge 2^n$.
Consider $\psi = \sum\limits_{n=1}^{\infty} {\displaystyle \frac
{\psi_n}{2^n}}$. Then $0 \le \psi \le \varphi$, while the series
$\sum\limits_{n=1}^{\infty} {\displaystyle F_a \bigl( \frac
{\psi_n}{2^n}\bigr)}$ does not converge, a contradiction.

Thus, $F_a \in Y^*$.

Now, consider the mapping $\gamma$ which is the isometric and order
isomorphism of $L_1(\varphi)^{\text{sa}}$ onto
$\mathcal{M}^{\text{sa}}_*$ (see \cite[Theorem 17.1 and Theorem
17.6]{She1}, \cite[Theorem 2]{TruShe}). Then $\gamma^{*}$ is the
isometric and order isomorphism of
$(\mathcal{M}^{\text{sa}}_*)^* = \mathcal{M}^{\text{sa}}$ onto
$(L_1(\varphi)^{\text{sa}})^* = Y$ and $\gamma^{**}$ is the
isometric and order isomorphism of $Y^*$ onto
$(\mathcal{M}^{\text{sa}})^*$.

Let us show that the functional $\gamma^{**}(F_a)$ on
$\mathcal{M}^{\text{sa}}$ is $\sigma$-weakly continuous.
Take $x_n$, $x$ in $\mathcal{M}^+$ such that
$x = \sum\limits_{n=1}^\infty x_n$ in the sense of $\sigma$-weak
topology on $\mathcal{M}^{\text{sa}}$, that is equivalent to
$x = \sup\limits_k \sum\limits_{n=1}^k x_n$. Then
$\gamma^* (x) = \sum\limits_{n=1}^\infty \gamma^* (x_n )$ and we
have by (3):
$$
\gamma^{**}(F_a)(x) = F_a (\gamma^* (x))
= \sum\limits_{n=1}^\infty F_a (\gamma^* (x_n ))
= \sum\limits_{n=1}^\infty \gamma^{**}(F_a)(x_n) .
$$
It follows (cf. \cite[Corollary III.3.11]{Ta1}) that
$\gamma^{**}(F_a)$ is $\sigma$-weakly continuous, i.\,e. belongs to
$\mathcal{M}^{\text{sa}}_*$. Therefore we can consider the
integrable sesquilinear form $\gamma^{-1}(\gamma^{**}(F_a))$ which
coincides with $a$ by uniqueness arguments.
\end{proof}

\begin{remark}
In the general case of infinite weight the validity of the
implication $\mathrm{(ii)} \implies \mathrm{(i)}$ question remains
open. However, it follows from results of \cite{DoShe} that the
implication holds in the special case of standard trace on the
algebra $\mathcal{B}(H)$ of all bounded operators on a Hilbert space
$H$ (see also \cite[Theorem 5.2]{She1}).
\end{remark}

\end{document}